\documentclass[12pt,reqno]{article}
\usepackage{amssymb,amscd,amsmath,amsthm,color}
\newtheorem{theorem}{Theorem}[section]

\newtheorem{cor}[theorem]{Corollary}
\newtheorem{prop}[theorem]{Proposition}

\usepackage{enumerate,amssymb}
\theoremstyle{definition}

\theoremstyle{remark}

\numberwithin{equation}{section}

\def\bC{\mathbb{C}}

\def\bM{\mathbb{M}}

\def\bR{\mathbb{R}}

\def\bR{\mathbb{R}}

\begin{document}
\baselineskip=15pt

\title{Positive linear maps on normal matrices}

\author{ Jean-Christophe Bourin{\footnote{This research was supported by the French Investissements d’Avenir program, project ISITE-BFC (contract ANR-15-IDEX-03).
}} \ and  Eun-Young Lee{\footnote{This research was supported by
Basic Science Research Program through the National Research
Foundation of Korea (NRF) funded by the Ministry of
Education (NRF-2018R1D1A3B07043682)}  }}

\date{ }

\maketitle

\vskip 10pt\noindent
{\small
{\bf Abstract.}  {  For a positive linear map $\Phi$ and a normal matrix $N$, we show that $|\Phi(N)|$ is bounded by some simple linear combinations in the unitary orbit of $\Phi(|N|)$. Several elegant sharp inequalities are derived, for instance for the Schur product of two normal matrices $A,B\in\bM_n$,
$$
|A\circ B| \le |A|\circ|B|+ \frac{1}{4}V(|A|\circ|B|)V^*
$$
for some unitary $V\in\bM_n$, where the constant $1/4$ is optimal.}
\vskip 5pt\noindent
{\it Keywords.} Matrix inequalities,  Positive linear maps, Schur products, unitary orbits, matrix geometric mean.
\vskip 5pt\noindent
{\it 2010 mathematics subject classification.} 47A30, 15A60.
}

\section{Introduction } 

Let $\bM_n$ denote the space of complex $n\times n$ matrices and let $\bM_n^+$ stands for the positive (semi-definite) cone.
A linear map $\Phi:\bM_n\to\bM_m$ is called positive if $\Phi(\bM_n^+)\subset \bM_m^+$. A nice exposition of these maps on matrix algebras can be found in Bhatia's  book \cite{Bhatia2}.
The Schur product with $A\in\bM_n^+$, $X\mapsto A\circ X $ is a basic example of positive linear map on $\bM_n$ (Schur's theorem; for which it suffices to check that the Schur product of two rank one positive matrices is again a rank one positive matrix).

A famous Schwarz inequality due to Man-Duen Choi \cite{Choi}, and previously to Kadison  for the Hermitian case,  says that
$$
\Phi(N^*)\Phi(N) \le  \Phi(N^*N)
$$
for all normal matrices $N$ in a unital matrix algebras ${\mathcal{A}}$ on which acts the  unital positive linear mas $\Phi$. Written as 
$|\Phi(N)|^2\le \Phi(|N|^2)$, this is a  noncommutative Jensen inequality for the convex function  $z\mapsto |z|^2$ on $\bC$. Choi's inequality is related  to another classical fact,  the Russo-Dye theorem \cite{RD} stating that unital positive linear maps are contractive, or more generally, that any positive linear maps attains its norm at the identity.

Comparing $|\Phi(N)|$ and $\Phi(|N|)$ via operator inequalities  seems more delicate.  In this note, we aim to obtain such comparisons by using additive combination in the unitary orbits of  $\Phi(|N|)$.  As a consequence we point out some refinements of the classical Russo-Dye theorem stating that any positive linear maps attains its norm at the identity. The results may also be useful to obtain some estimates for
various non $\bC$-linear maps; we close this introduction by providing such an example. For $A = (a_{i,j})\in\bM_n$, we consider the conjugate matrix $\overline{A} := A^{*T} = (\overline{a_{i,j}} )$, and the
matrix with the real parts, $A_{\bR} := (A + \overline{A}/2$.

\begin{prop} If A is a normal matrix, then $|\det A_{\bR}| \le  \det |A|_{\bR}$.
\end{prop}

This is a consequence of Corollary 2.4. This inequality does not hold for some
nonnormal two-by-two matrices.

\section{Positive linear maps}

We first single out the simplest case of our main theorem.

\begin{prop}   Let $\Phi: \bM_n\to \bM_m$ be  a positive linear map and let $N\in \bM_n$ be normal.  Then, there exists  a unitary $V\in\bM_m$ such that
$$
   |\Phi(N)|  \le \frac{\Phi(|N|) +V\Phi(|N|)V^*}{2}.
$$
\end{prop}

Thus, $|\Phi(N)|$ is dominated by the mean of two elements in the unitary orbit of
$\Phi(|N|)$. There exists a whole family of such sharp inequalities, this is the content
of the following main theorem.  $R\in\bM_n$ is called a reflexion if $R^2=I$.

\vskip 5pt
\begin{theorem}\label{mainthm}     Let $\Phi: \bM_n\to \bM_m$ be  a positive linear map, $n,m\neq1$,  and let $N\in \bM_n$ be normal. Fix $\beta> 0$. Then, there exists  a unitary $V\in\bM_m$ such that
$$
   |\Phi(N)|  \le \beta\Phi(|N|) +\frac{1}{4\beta}V\Phi(|N|)V^*.
$$
If $\beta \ge 1/2$, then the constant $1/4\beta$ is the smallest possible one and this  inequality is sharp even if we confine $N$ to the class of Hermitian reflexions in $\bM_n$. 

A refinement of this inequality is the double inequality
$$
   |\Phi(N)| \le  \Phi(|N|)\# V\Phi(|N|)V^* \le \beta\Phi(|N|) +\frac{1}{4\beta}V\Phi(|N|)V^*.
$$
where $\#$ stands for the usual geometric mean.
\end{theorem}

\vskip 5pt
We do not know wether the theorem is sharp or not when $0<\beta<1/2$.
 A reader familiar with the theory of positive maps will note that, in the course of the proof, we quickly establish a basic fact due to Stinespring \cite{Stinespring} stating that  positive maps defined on a commutative domain are completely positive. 

\vskip 5pt
\begin{proof} Consider the spectral decomposition of $N$,
$$
N=\sum_{i=1}^n \sigma_i E_i
$$
where $\sigma_i$ runs over the eigenvalues of $N$ counted with their multiplicities and $E_i$ are 
 rank one projections, $i=1,\ldots, n$. Note that $E_i=x_i x_i^*$  for some column vectors $x_i\in \bM_{n,1}$. Let  ${\mathcal{N}}$ be the $*$-commutative subalgebra   spanned by the projections $E_i$ ,$i=1,\ldots, n$.
 When restricted to  ${\mathcal{N}}$, the map $\Phi$ has the form
\begin{equation}\label{decomp1}\Phi(X)=\sum_{i=1}^m\sum_{j=1}^n Z_{i,j}^* XZ_{i,j}\end{equation}
for some rank 1 or 0 matrices $Z_{i,j}\in\bM_{n,m}$, $i=1,\ldots,n$, $j=1,\ldots,m$. Indeed, the map $\Phi$ on ${\mathcal{N}}$ is determined by its values on the projections $E_i$, hence, setting
$Z_{i,j}= x_i R_{i,j}$ where $R_{i,j}\in\bM_{1,m}$ is the $j$-th row of $\Phi(E_i)^{1/2}$, we obtain the representation \eqref{decomp1}. 

This decomposition of $\Phi$ on ${\mathcal{N}}$ as a sum of congruences maps ensures
that
$$
\begin{pmatrix}
\Phi(A)& \Phi(B)\\
\Phi(B^*)& \Phi(C)
\end{pmatrix} \ge 0
$$
whenever the matrices $A,B,C$ in ${\mathcal{N}}$ satisfy
$$
\begin{pmatrix}
A& B\\
B^*& C
\end{pmatrix} \ge 0.
$$
In particular, this happens if $A=C=|N|$ and $B=N$, as
$$
\begin{pmatrix}
|z| &z \\
z^*&|z|
\end{pmatrix}\ge 0 
$$
for all $z\in\bC$.
Therefore
\begin{equation}\label{eqfund}
\begin{pmatrix}
\Phi(|N|)& \Phi(N)\\
\Phi(N^*)& \Phi(|N|)
\end{pmatrix} \ge 0.
\end{equation}
Now, let $V^*$ be the unitary part in the polar decomposition $\Phi(N)=V^*|\Phi(N)|$. We have, for all $\alpha>0$,
\begin{equation*}
\begin{pmatrix} \alpha^{-1/2} V & -\alpha^{1/2} I \end{pmatrix}
\begin{pmatrix}
\Phi(|N|)& \Phi(N)\\
\Phi(N^*)& \Phi(|N|)
\end{pmatrix} 
\begin{pmatrix} \alpha^{-1/2} V^*\\ -\alpha^{1/2} I \end{pmatrix}
\ge 0.
\end{equation*}
Equivalently,
$$
   |\Phi(N)| \le \frac{\alpha\Phi(|N|) +\alpha^{-1}V\Phi(|N|)V^*}{2}
$$
and setting $\beta=\alpha/2$ yields 
the inequalities of the lemma.

A map $\Phi:\bM_n\to\bM_m$ induces, by taking compressions to the upper left $2\times 2$ corners, a map from $\bM_2$ to $\bM_2$. Therefore, to check that the constant $1/4\beta$ is optimal when $\beta \ge 1/2$, we may assume $n=m=2$. Thus suppose $\beta\ge 1/2$, consider
$$
A=\begin{pmatrix}
2\beta& 1\\
1& (2\beta)^{-1}
\end{pmatrix}
$$
and the positive linear map on $\bM_2$, $Z\mapsto A\circ Z$, and take the reflexion 
$$
R=\begin{pmatrix}
0& 1\\
1& 0
\end{pmatrix}.
$$
We claim that the inequality, for any unitary $V\in\bM_2$ and any $0<c<1$,
\begin{equation}\label{optimal}
|A\circ R| -\beta (A\circ |R|) \le \frac{c}{4\beta} V (A\circ |R|) V^*
\end{equation}
cannot hold, so that the constant $1/4\beta$ is optimal in our theorem. Indeed, using $\beta \ge 1/2$, we have 
$$
\lambda_1(|A\circ R| -\beta (A\circ |R|) )= \frac{1}{2}, \qquad \lambda_2(|A\circ R| -\beta (A\circ |R|) )=1-2\beta^2,
$$
and
$$
 \frac{c}{4\beta} \lambda_1 (A\circ |R|)= \frac{c}{2}, \qquad  \frac{c}{4\beta} \lambda_2 (A\circ |R|)= \frac{c}{8\beta^2}. 
$$
As \eqref{optimal} entails 
$$
\frac{1}{2}=\lambda_1(|A\circ R| -\beta (A\circ |R|) ) \le \frac{c}{4\beta} \lambda_1 (A\circ |R|)= \frac{c}{2}
$$
we necessarily have $c\ge 1$ in \eqref{optimal}.

To get the double inequality involving the geometric mean $\#$, consider again the block-matrix \eqref{eqfund} and apply a unitary congruence with
\begin{equation*}
\begin{pmatrix}
V& 0\\
0& I
\end{pmatrix}
\end{equation*}
where $V^*$ is still the unitary  factor in the polar decomposition of $\Phi(N)$.
 This shows that 
\begin{equation*}
\begin{pmatrix}
V\Phi(|N|)V^*& |\Phi(N)|\\
|\Phi(N^*)|& \Phi(|N|)
\end{pmatrix} \ge 0.
\end{equation*}
It then follows from the maximal property of $\#$ that
$$
|\Phi(N)| \le V\Phi(|N|)V^*\# \Phi(|N|)
$$
which is exactly the first inequality in the double inequality. The second one is a simple application of the arithmetic-geometric means inequality.
\end{proof}

\vskip 5pt
Several eigenvalue inequalities follow from the theorem combined with some basic relations of Weyl type.

\begin{cor}\label{corfolk}
 Let $\Phi: \bM_n\to \bM_m$ be  a positive linear map, $n,m\neq1$,  and let $N\in \bM_n$ be normal. Then, for all $\beta >0$ and all integers $j,k\ge 1$,
\begin{itemize}
\item[(a)] $|\Phi(N)| \prec_{w\log}  \Phi(|N|)$
\item[(b)]  $\lambda_{j+k-1}(|\Phi(N)|) \le \sqrt{\lambda_{j}(\Phi(|N|))\lambda_{k}(\Phi(|N|))}$
\item[(c)] $4\beta\lambda_{j}\left\{|\Phi(N)|-\beta\Phi(|N|)\right\} \le \lambda_{j}(\Phi(|N|))$.
\end{itemize}
\end{cor}

\vskip 5pt
The weak log-majorisation (a) is well-known, see \cite[Theorem 2.9]{BL1} for a much more general result. (a) and (b) easily follow from the first inequality in the double inequality of the theorem, however (a) and (b) can also be  derived from \eqref{eqfund} and Horn's inequality. (c) is a more subtle estimate, it is a restatement of the first inequality of the theorem.
The next corollary extends Proposition 1.1.

\begin{cor} If A is a normal matrix, then $| A_{\bR}| \prec_{w\log}  |A|_{\bR}$.
\end{cor}

\begin{proof} Observe that, for any polynomial $p(t)$ with only real coefficients and any
$T \in\bM_n$,
$$
p(\overline{T}^*\overline{T}) = \overline{p(T ^*T )}
$$
and so $|\overline{T}| = \overline{|T |}$. Now, consider the normal operator
$$
N=\begin{pmatrix} A&0 \\ 0&\overline{A}\end{pmatrix}.
$$
By the previous observation,
$$
| N|=\begin{pmatrix} |A|&0 \\ 0&\overline{|A|}\end{pmatrix},
$$
and applying Corollary 2.3 to $N$ with the positive linear map,
$$
\begin{pmatrix} X&R \\ L&Y\end{pmatrix}\mapsto \frac{X+Y}{2},
$$
yields  $| A_{\bR}| \prec_{w\log}  |A|_{\bR}$.
\end{proof}

Some quite useful positive linear maps are  the partial traces on tensor products. Identifying $ \bM_{d}\otimes \bM_{n}$ with $\bM_{d}(\bM_n)$, the partial trace for the first factor 
 ${\mathrm{Tr}}_{\bf{1}}:  \bM_{d}\otimes \bM_{n}\to \bM_n$ is defined as
$$
{\mathrm{Tr}}_{\bf{1}}\left(
\begin{pmatrix}
A_{1,1} &\cdots& A_{1,d} \\
\vdots & \ddots &\vdots \\
A_{d,1}& \cdots & A_{d,d}
\end{pmatrix}\right)
=\sum_{i=1}^d A_{i,i}.
$$

\vskip 5pt
\begin{cor} Let $N\in \bM_{d}\otimes \bM_{n}$ be normal. Then, for some unitary $V\in\bM_n$,
$$
|{\mathrm{Tr}}_{\bf{1}}(N)| \le {\mathrm{Tr}}_{\bf{1}}(|N|)\# V{\mathrm{Tr}}_{\bf{1}}(|N|)V^*
$$
\end{cor}

Applying this corollary to a block-diagonal normal matrix, $N=N_1\oplus\cdots \oplus N_d$, we obtain the inequality: if $N_1,\ldots, N_d$ are normal matrices in $\bM_n$, then for some unitary $V\in\bM_n$,
$$
 \left|\sum_{i=1}^d N_i\right| \le  \left\{\sum_{i=1}^d|N_i|\right\} \#\left\{ V\left(\sum_{i=1}^d|N_i|\right)V^*\right\}.
$$
and, consequently,
$$
\left|\sum_{i=1}^d N_i\right| \le \frac{1}{2}\left\{ \sum_{i=1}^d|N_i| +V\left(\sum_{i=1}^d|N_i|\right)V^*\right\}.
$$
Though the geometric mean inequality is stronger than the arithmetic mean one, the second one may be more natural and useful. For instance, combining with some well-known eigenvalue inequalities for convex functions  we may infer from the previous inequality \cite[Proposition 2.11]{BL0}.

Let us mention that Theorem \ref{mainthm} is optimal: in general, $|\Phi(N)|$ cannot be bounded by a single element into the unitary orbit of $\Phi(|N|)$. In fact, for any $c>0$, there is no general inequality of the form 
\begin{equation}\label{eqc}
|\Phi(N)| \le c V \Phi(|N|)V^*
\end{equation}
for some unitary $V$.
 To check that, take $\Phi$ as the partial trace of $\bM_2(\bM_2)$ and set 
$$
N=  \begin{pmatrix}  P & 0 \\ 0& - Q\end{pmatrix} 
$$
where $P,Q\in\bM_2$ are the projections
$$
P=  \begin{pmatrix}  1 & 0 \\ 0& 0\end{pmatrix}, \qquad
Q=
\begin{pmatrix}  \cos^2 a & \sin a\cos a \\ \sin a\cos a& \sin^2 a \end{pmatrix}.
$$
Since $$
\lim_{a\to 0^+} \frac{\lambda_2(\Phi(|N|))}{\lambda_2(|\Phi(N)|)}=\lim_{a\to 0^+} \frac{\lambda_2(P+Q)}{\lambda_2(|P-Q|)} = \lim_{a\to 0^+}  \frac{1-\cos a}{\sin a}= 0
$$
we infer that \eqref{eqc} cannot hold.

\section{Some special cases}

The Russo-Dye theorem \cite{RD} says that, given a positive linear map $\Phi$ on a unital $C^*$-algebra ${\mathcal{A}}$, we have $\| \Phi(Z)\|\le \|\Phi(I)\|$ for all contractions $Z\in{\mathcal{A}}$. We can say much more as stated in the next corollary. We stay in the setting of $\bM_n$, but the same result holds for an arbitrary  unital $C^*$-algebra ${\mathcal{A}}$ acting on a separable Hilbert space ${\mathcal{H}}$, by considering partial isometries  instead of unitaries in $\bM_m$.

\vskip 5pt
\begin{cor}\label{corcontraction}  Let $Z\in \bM_n$ be a contraction and  let $\Phi: \bM_n\to \bM_m$ be  a positive linear map. Then, for some unitary $V\in\bM_m$,
$$
   |\Phi(Z)| \le \frac{\Phi(I) +V\Phi(I)V^*}{2}.
$$
\end{cor}

\vskip 5pt
\begin{proof} We may dilate $Z$ into a unitary $U\in\bM_2(\bM_n)$, for instance with Halmos,
$$
U=\begin{pmatrix}
Z&-\sqrt{I-ZZ^*} \\
\sqrt{I-Z^*Z}& Z^*
 \end{pmatrix}
$$
Now, let $\Psi:\bM_2(\bM_n)\to\bM_m$ be defined as
$$
\Psi\left( \begin{pmatrix} A&B \\ C&D\end{pmatrix}\right)=\Phi(A).
$$
Applying Theorem \ref{mainthm} to $\Psi$ and $U$, we have
$$
|\Phi(Z)|=|\Psi(U)|\le \frac{\Psi(|U|) +V\Psi(|U|)V^*}{2}= \frac{\Phi(I) +V\Phi(I)V^*}{2}
$$
for some unitary $V\in\bM_m$.
\end{proof}

Arguing as in the previous proof, and using Corollary \ref{corfolk} (a) we have the following log-majorisation:
\vskip 5pt
\begin{cor}\label{maincor2}  Let $Z\in \bM_n$ be a contraction and  let $\Phi: \bM_n\to \bM_m$ be  a positive linear map. Then
$$
   |\Phi(Z)| \prec_{w\!\log}  \Phi(I).
$$
\end{cor}

\vskip 5pt
By using a block diagonal matrix $Z=Z_1\oplus \cdots \oplus Z_k$ and an obvious positive linear map,  Corollary \ref{maincor2} reads as a folklore result:

\vskip 5pt
\begin{cor}\label{corsurp} Let $X_1,\ldots, X_k$ be in $\bM_{m,n}$ and let $Z_1,\ldots, Z_k$  be contractions in $\bM_n$. Then,  
$$
\left|\sum_{i=1}^k X^*_iZ_iX_i\right| \prec_{w\!\log}\sum_{i=1}^k X^*_iX_i.
$$
\end{cor}

\vskip 5pt
In fact, Corollary \ref{corsurp} is rather trivial for an expert since the positivity 
$$
\begin{pmatrix} \sum_{i=1}^k X^*_i|Z_i^*|X_i& \sum_{i=1}^k X^*_iZ_iX_i \\ \sum_{i=1}^k X^*_iZ_i^*X_i & \sum_{i=1}^k X^*_i|Z_i|X_i\end{pmatrix}\ge 0
$$
entails the factorisation
$$
\sum_{i=1}^k X^*_iZ_iX_i= \left( \sum_{i=1}^k X^*_i|Z_i^*|X_i\right)^{1/2} K \left(\sum_{i=1}^k X^*_i|Z_i|X_i \right)^{1/2}
$$
for some contraction $K\in\bM_n$, so that Corollary \ref{corsurp} follows from Horn's inequalities.

\vskip 5pt
A special case of Corollary \ref{corcontraction} stresses on the role of  diagonals  for the Schur product with a positive matrix.

\vskip 5pt
\begin{cor}\label{corschur1}  Let $A\in\bM_n^+$ and let $Z\in \bM_n$ be a contraction.  If $D$ denotes the diagonal part of $A$, then for some unitary $V\in\bM_n$,
$$
   |A\circ Z|  \le \frac{D+VDV^*}{2}
$$
\end{cor}

\vskip 5pt
Of course, in Corollary \ref{corcontraction} and Corollary  \ref{corschur1} we may use the geometric mean instead of the arithmetic mean.

Another application to Schur product may be recorded.

\vskip 5pt
\begin{cor}\label{corschurnormal} If $A,B\in\bM_n$ are normal,  then, for some unitary $V\in\bM_n$,
\begin{equation*}
|A\circ B| \le |A|\circ|B|+ \frac{1}{4}V(|A|\circ|B|)V^*
\end{equation*}
where the constant $1/4$ is optimal.
\end{cor}

\vskip 5pt
\begin{proof} Note that $A\circ B= \Phi(A\otimes B)$ for some positive linear map (the extraction of a principal submatrix). As $A\otimes B$ is normal and $|A\otimes B|=|A|\otimes |B|$, the inequality follows from Theorem \ref{mainthm} with $\beta=1/4$. The sharpness of $1/4$  is established in  \eqref{optimal}.
\end{proof}

\vskip 5pt
The next consequence of our theorem is of a rather general nature, and we state it with the geometric mean.

\vskip 5pt
\begin{cor}\label{corfinal}     Let $\Phi: \bM_n\to \bM_m$ be  a positive linear map  and let $X\in \bM_n$. Then, for some unitary matrices $U,V\in\bM_m$,
$$
   |\Phi(X+X^*)| \le \Phi(|X|+|X^*|) \#U\Phi(|X|+|X^*|)U^*
$$
and
$$
 |\Phi(X\circ X^*)| \le \Phi(|X|\circ|X^*|) \#V\Phi(|X|\circ|X^*|)V^*.
$$
\end{cor}

\vskip 5pt
\begin{proof} Let $\Psi:\bM_2(\bM_n)\to\bM_m$ be defined as
$$
\Psi\left( \begin{pmatrix} A&B \\ C&D\end{pmatrix}\right)=\Phi(A+B+C+D).
$$
Since
$$
A+B+C+D=\begin{pmatrix} I&I\end{pmatrix}\begin{pmatrix} A&B \\ C&D\end{pmatrix}\begin{pmatrix} I\\ I\end{pmatrix},
$$
 $\Psi$ is a positive map. Applying Theorem  \ref{mainthm}   to this map with the normal (Hermitian) operator in $\bM_2(\bM_n)$
$$
\begin{pmatrix} 0&X \\ X^*&0\end{pmatrix}
$$
yields the first inequality. 

To get the second one,  use again that $A\circ B= \Gamma(A\otimes B)$ for some positive linear map $\Gamma$ and pick for $A$ and $B$
 the pair of Hermitian operators in $\bM_{2n}$,
$$
A=\begin{pmatrix} 0&X^* \\ X&0\end{pmatrix}, \quad B=\begin{pmatrix} 0&X \\ X^*&0\end{pmatrix}.
$$
Apply Theorem  \ref{mainthm} to the normal matrix $A\otimes B$ and the positive linear map $\Phi=\Lambda \circ \Gamma$,  where $\Lambda$ is the map on $\bM_{2n}=\bM_2(\bM_n)$ defined as
$$
\Lambda\left(\begin{pmatrix} S&T \\ Q&R\end{pmatrix}\right)= S+T+Q+R.
$$
This yields the second inequality of the Theorem.
\end{proof}

Our last illustration of our main result deals with a special class of positive linear maps whose definition will be recalled at the beginning of the proof.

\vskip 5pt
\begin{cor}\label{lastcor}     Let $\Phi: \bM_n\to \bM_m$ be  a unital two-positive linear map and let $Z\in\bM_n$ be a contraction. Then, 
$$
   |\Phi(Z)| \le  \frac{1}{4}I +\Phi(|Z|).
$$
If $n\ge 3$ and $m\ge 2$, then  the constant $1/4$ is sharp, even if we confine $\Phi$ to the class of unital completely positive maps and  $Z$ to the  set of Hermitian contractions.
\end{cor}

\vskip 10pt
\begin{proof} $\Phi$ is unital means $\Phi(I)=I$. That $\Phi$ is two-positive means that the  linear map on $\bM_2(\bM_n)$,
$$
\begin{pmatrix}
A & B \\ C & D
\end{pmatrix} \mapsto 
\begin{pmatrix}
\Phi(A) & \Phi(B) \\ \Phi(C) & \Phi(D)
\end{pmatrix}
$$
is positive. Applying Theorem \ref{mainthm}
with $\beta=1/4$ to this map and a normal (Hermitian) operator in $\bM_2(\bM_n)$ of the form
$$\begin{pmatrix}0 & Z \\ Z^*& 0
\end{pmatrix}
$$
we obtain
$$
\begin{pmatrix}
|\Phi(Z^*) |& 0 \\ 0 & |\Phi(Z)|
\end{pmatrix}
\le 
\begin{pmatrix}
\Phi(|Z^*|) & 0 \\ 0 & \Phi(|Z|)
\end{pmatrix}
+\frac{1}{4} V \begin{pmatrix}
\Phi(|Z^*|) & 0 \\ 0 & \Phi(|Z|)
\end{pmatrix} V^*
$$
for some unitary $V\in\bM_2(\bM_n)$. Assuming further that $\Phi$ is (sub)unital  we infer the inequality of the theorem,
\begin{equation}\label{best}
|\Phi(Z)| \le \frac{1}{4}I + \Phi(|Z|)
\end{equation}
for all contractions $Z\in\bM_n$.

To check that $1/4$ is the best possible in \eqref{best}, consider the map $\Psi:\bM_3\to \bM_2$,
\begin{equation}\label{mapPsi}
T=\begin{pmatrix}
t_{1,1} & t_{1,2} & t_{1,3} \\
t_{2,1} & t_{2,2} & t_{2,3} \\
t_{3,1} & t_{3,2} & t_{3,3} 
\end{pmatrix}
\mapsto \Psi(T)=
\begin{pmatrix} t_{1,1} & \frac{1}{2} t_{1,2} \\ \frac{1}{2} t_{2,1} & \frac{1}{4} t_{2,2} + \frac{3}{4} t_{3,3}
\end{pmatrix}.
\end{equation}
This is a completely positive map as
$$
\begin{pmatrix} t_{1,1} &  t_{1,2} \\ t_{2,1} & t_{2,2}
\end{pmatrix}\mapsto
\begin{pmatrix} t_{1,1} &  \frac{1}{2}t_{1,2} \\ \frac{1}{2} t_{2,1} & \frac{1}{4} t_{2,2}
\end{pmatrix}
$$
is a Schur product with a matrix in $\bM_2^+$.
With the Hermitian matrix \begin{equation}\label{hermitianmatrixR}
H=\begin{pmatrix}
0 & 1 & 0 \\
1 & 0& 0 \\
0 & 0 &0 
\end{pmatrix}
\end{equation}
we have
\begin{equation}\label{eqfin}
|\Psi(H)| =\frac{1}{4}\begin{pmatrix} 0 & 0 \\ 0 & 1\end{pmatrix} + \Psi(|H|)
\end{equation}
and so $1/4$ cannot be replaced by anything smaller in \eqref{best}.
\end{proof}

\vskip 5pt
\noindent
Laboratoire de math\'ematiques, 

\noindent
Universit\'e de Bourgogne Franche-Comt\'e, 

\noindent
25 000 Besan\c{c}on, France.

\noindent
Email: jcbourin@univ-fcomte.fr

  \vskip 10pt
\noindent Department of mathematics, KNU-Center for Nonlinear
Dynamics,

\noindent
Kyungpook National University,

\noindent
 Daegu 702-701, Korea.

\noindent Email: eylee89@knu.ac.kr

\end{document}